\documentclass[11pt,
]{amsart}
\usepackage{
amssymb, graphicx
}
 \usepackage[foot]{amsaddr}
\usepackage{amsthm}
\usepackage{graphicx,color}
\newtheorem{theorem}{Theorem}
\newtheorem{lemma}{Lemma}
\newtheorem{proposition}{Proposition}

\newtheorem{remark}{Remark}

\theoremstyle{remark}

\newcommand{\be}[1]{\begin{equation}\label{#1}} 
\newcommand{\ee}{\end{equation}} 

\newcommand{\p}{\partial}

\DeclareMathOperator{\vnabla}{\stackrel{{\tiny \mathtt{v}}}{\nabla}}
\DeclareMathOperator{\hnabla}{\stackrel{{\tiny \mathtt{h}}}{\nabla}}
\DeclareMathOperator{\vdiv}{\stackrel{{\tiny \mathtt{v}}}{\mathrm{div}}}
\DeclareMathOperator{\hdiv}{\stackrel{{\tiny \mathtt{h}}}{\mathrm{div}}}

\title[Notes on tensor tomography]{Notes on tensor tomography on Riemannian manifolds with boundaries}

  \author[J. Zhai]{Jian Zhai}
\address{J. Zhai: Institute for Advanced Study,
  The Hong Kong University of Science and Technology, Kowloon, Hong Kong, China
  (\tt{iasjzhai@ust.hk}).}

\begin{document}

\begin{abstract}
We summarize the proofs for the $s$-injectivity of the tensor tomography problem on compact Riemannian manifolds with boundaries in [Dairbekov, Inverse Problems, 22: 431, 2006] and [Paternain-Salo-Uhlmann, Math. Ann., 363: 305-362, 2015] under certain geometric assumptions.
\color{black}
\end{abstract}

\maketitle 
\section{Introduction}
Let $(M,g)$ be a compact simple $d$-dimensional Riemannian manifold with boundary $\partial M$. Recall that a compact Riemannian manifold with boundary is called simple if the boundary $\partial M$ is strictly convex, and the exponential map $\exp_p:T_pM\supset\exp_p^{-1}(M)\rightarrow M$ is a diffeomorphism for any $p\in M$.
We use
the notation $TM$ for the tangent bundle of $M$,
\color{black}
and  $S M$ for the unit sphere bundle, defined as
$SM=\{(x,v)\in TM;\,|v|_g=1\}$. Denote $\partial_+(SM)=\{(x,v)\in
SM;\, x\in\partial M,\,\langle v,\nu\rangle_g < 0\}$, where $\nu$ is the unit outer normal vector field to the boundary. 

The geodesic ray transform $I_mf $
of a smooth symmetric $m$-tensor $f$ is given by the following formula
\begin{equation}
\label{raytransform}
I_mf(x,v)=\int_{0}^{\tau(x,v)}  f_{i_1\cdots i_m}(\gamma(t))\dot\gamma(t)^{i_1}\cdots\dot\gamma(t)^{i_m}\mathrm{d}t,
\end{equation}
where $(x,v)\in\p_+(SM)$ and $\gamma(t)=\gamma_{x,v}(t)$ is the unit-speed geodesic
given by the initial conditions $(\gamma(0),\dot{\gamma}(0))=(x,v)$. Here $\tau(x,v)$ is the time when the geodesic $\gamma_{x,v}$ exits the manifold $M$.

The geodesic ray transform for the case $m=2$ arises from the linearization of the \textit{boundary rigidity} and \textit{lens rigidity} problem. We refer to \cite{merry2011inverse} for detailed explanations. If one only considers the rigidity problems in the same conformal class, the problem reduces to the case $m=0$. For the case $m=1$, an important application is the ultrasound Doppler tomography \cite{juhlin1992principles}.
The case $m=4$ is related with the elastic \textit{qP}-wave travel-time tomography in weakly anisotropic medium \cite{vcerveny1982linearized,Shara}. See \cite{paternain2014tensor,ilmavirta2019integral} for an overview of tensor tomography problems.\\

It is clear that if $h$ is a symmetric $(m-1)$-tensor with $h\vert_{\partial M}=0$, then $I_m(\mathrm{d}^sh)=0$. Here $\mathrm{d}^s=\sigma\nabla$, where $\nabla$ is the Levi-Civita connection with respect to $g$ and $\sigma$ is the symmetrization. The operator $I_m$ is said to be $s$-injective if $I_mf=0$ implies that $f=\mathrm{d}^sh$ for some $h$ satisfying $h\vert_{\partial M}=0$.

The $s$-injectivity of $I_m$ has been studied extensively under the simplicity assumptions on $(M,g)$. For $d\geq 2$, the injectivity of $I_0$ was established in \cite{Muk2,Muk1}, and $s$-injectivity of
  $I_1$ in \cite{AR}. In dimension $d=2$, it was proved in \cite{paternain2013tensor} that $I_m$ is $s$-injective for arbitrary $m$, while for the case $m=2$ it was proved earlier in \cite{Shara1} following the proof for the boundary rigidity problem in \cite{PU2}. In higher dimensions $d\geq 3$, the problem for $m\geq 2$ remains to be open. It was however shown in \cite{SU3} that the $s$-injectivity of $I_2$ is true for a generic simple metric $g$. Under stronger geometric conditions, such as an explicit upper bound of the sectional curvature, the $s$-injectivity is also proved \cite{PS,Shara,dairbekov2006integral,paternain2015invariant}. We will summarize the work in \cite{dairbekov2006integral,paternain2015invariant}, which establish the $s$-injectivity under some condition sufficiently implied by an explicit curvature bound \cite[Proposition 1.4]{dairbekov2006integral}. Let us introduce such condition in the following.

Following \cite{pestov2003well,paternain2015invariant}, we say that $(M,g)$ is free of $\beta$-conjugate points if, for any geodesic $\gamma$ and $J$ a nonzero vector field along $\gamma$ satisfying the equation
\begin{equation}\label{jacobi_eq}
D_t^2J+\beta R(J,\dot{\gamma})\dot{\gamma}=0
\end{equation}
with $J(0)=0$, then $J$ only vanishes at $t=0$. Here $D_t$ is the covariant derivative and $R$ is the Riemann curvature tensor. It is worth mentioning that if $(M,g)$ is free of $\beta_0$-conjugate points, it is also free of $\beta$-conjugate points for any $\beta\in [0,\beta_0)$. Note that a simple manifold is free of $1$-conjugate points, since the equation \eqref{jacobi_eq} with $\beta=1$ is the usual Jacobi equation. If $(M,g)$ has non-positive curvature, then it is free of $\beta$-conjugate points for any $\beta\geq 0$. See \cite{paternain2015invariant} for more discussions.

Let $\beta^{(1)}_{d,m}=\frac{m(d+m)}{d+2m-1}$ and $\beta^{(2)}_{d,m}=\frac{m(d+m-1)}{d+2m-2}$. Note that $\beta^{(1)}_{d,m}\geq \beta^{(2)}_{d,m}\geq 1$, where the equalities only hold when $m=1$. The following theorem is proved in \cite{dairbekov2006integral}:
\begin{theorem}{\cite[Theorem 1.1]{dairbekov2006integral}}\label{thm1}
Let $d\geq 2$, $m\geq 1$. Suppose $(M,g)$ is a compact simple Riemannian manifold free of $\beta^{(1)}_{d,m}$-conjugate points. Then $I_m$ is $s$-injective.
\end{theorem}
Then in \cite{paternain2015invariant}, the result is improved to:
\begin{theorem}{\cite[Theorem 1.6]{paternain2015invariant}}\label{thm2}
Let $d\geq 2$, $m\geq 1$. Suppose $(M,g)$ is a compact simple Riemannian manifold free of $\beta^{(2)}_{d,m}$-conjugate points. Then $I_m$ is $s$-injective.
\end{theorem}

The purpose of these notes is to provide a concise and unified presentation of the proofs in the two articles \cite{dairbekov2006integral,paternain2015invariant}. \\

Another powerful method to study the geodesic ray transform in dimension $d\geq 3$ is introduced by Uhlmann and Vasy \cite{UV}, where they proved that a modified version of the normal operator $I_0^*I_0$, localized near a convex boundary point, is elliptic as a ``scattering pseudodifferential operator". This yields the injectivity of $I_0$ under the condition that $(M,g)$ admits a strictly convex function. This condition is extensively studied (see \cite{paternain2019geodesic} and the references therein). The results were then generalized to case $m=1,2$ \cite{stefanov2018inverting} and $m=4$ \cite{de2019inverting}.
\section{Preliminaries}
In this section, we introduce the notations and fundamental integral identities needed for the proof of the main theorems.

It is known that any smooth symmetric $m$-tensor $f$ admits the \textit{Helmholtz decompostion} as
\[
f=f^s+\mathrm{d}^sp
\]
with $p\vert_{\partial M}=0$, and $\delta f^s=0$, where $\delta$ is the divergence operator ($-\delta$ is the formal adjoint of $\mathrm{d}^s$). The tensor fields $f^s$ and $\mathrm{d}^sp$ are called the \textit{solenoidal} and \textit{potential} parts of the tensor $f$. To show the $s$-injectivity of $I_m$, one only needs to show that $I_m$ is injective on solenoidal tensors. So from now on, we assume $\delta f=0$.

The tensor $f$ can be associated with a smooth function on $SM$, still denoted by $f$, in the following way
\[
f(x,v)=f_{i_1i_2\cdots i_m}(x)v^{i_1}v^{i_2}\cdots v^{i_m}.
\]
For any $(x,v)\in SM$, let $\varphi_t(x,v)=(\gamma_{x,v}(t),\dot{\gamma}_{x,v}(t))$ denote the geodesic flow of the Riemannian metric $g$ acting on $SM$. Define
\[
u(x,v)=\int_0^{\tau(x,v)}f(\varphi_t(x,v))\mathrm{d}t.
\]
Denote $X:C^\infty(SM)\rightarrow C^\infty(SM)$ to be the geodesic vector field associated with $\varphi_t$, i.e.,
\[
Xu(x,v)=\partial_tu(\varphi_t(x,v))\vert_{t=0}.
\]
Then if $I_mf=0$, we have
\begin{equation}\label{transport_eq}
Xu=-f,\quad\text{in }SM, \quad u\vert_{\partial(SM)}=0.
\end{equation}

The proof of Theorems \ref{thm1} and \ref{thm2} is based on the energy estimates for the transport equation \eqref{transport_eq} on the unit sphere bundle $SM$. We will use some concepts and notations introduced in \cite{paternain2015invariant} (see also \cite{ilmavirta2019integral} for more details). Consider the natural projection $\pi:\, SM\rightarrow M$. Let $\mathcal{V}$ denote the vertical subbundle of $TSM$ given by $\mathcal{V}=\mathrm{ker}\,\mathrm{d}\pi$. With respect to the Sasaki metric on $SM$, we have the orthogonal splitting of $TSM$ as
\[
TSM=\mathbb{R}X\oplus \mathcal{H}\oplus\mathcal{V}.
\]
For a smooth function $u\in C^\infty(SM)$, using the above splitting, we can write the gradient $\nabla_{SM} u$ as
\[
\nabla_{SM} u=((Xu)X,\,\hnabla u,\,\vnabla u).
\]
Denote by $\mathcal{Z}$ the set of smooth functions $Z:\, SM\rightarrow TM$ such that $Z(x,v)\in T_xM$ and $\langle Z(x,v),v\rangle=0$ for all $(x,v)\in SM$. Then $\vnabla u,\hnabla u\in\mathcal{Z}$. We have $X:C^\infty(SM)\rightarrow C^\infty(SM)$ and $\vnabla,\hnabla:C^\infty(SM)\rightarrow \mathcal{Z}$. We can also let $X$ act on $\mathcal{Z}$ as follows
\[
XZ(x,v):=\frac{DZ(\varphi_t(x,v))}{\mathrm{d}t}\Big\vert_{t=0}.
\]
Note that $X:\mathcal{Z}\rightarrow\mathcal{Z}$ as discussed in \cite{paternain2015invariant}. Let $N$ denote the subbundle of $\pi^*TM$, such that $N_{(x,v)}=\{v\}^\perp$. Then $\mathcal{Z}$ coincides with the smooth sections of the bundle $N$. Using the natural $L^2$-inner product on $N$, and the $L^2$-inner product on $C^\infty(SM)$, one can introduce the formal adjoints of $\vnabla$ and $\hnabla$ as $-\vdiv:\mathcal{Z}\rightarrow C^\infty(SM)$ and $-\hdiv:\mathcal{Z}\rightarrow C^\infty(SM)$. Also $X^*=-X$ whenever acting on $C^\infty(SM)$ or $\mathcal{Z}$. Let $R(x,v):\{v\}^\perp=\{v\}^\perp$ be the operator given by $R(x,v)w=R_x(w,v)v$, where $R_x$ is the Riemannian curvature tensor.\\

 The following commutator formulas are needed to derive a set of integral identities.
\begin{lemma}{\cite[Lemma 2.1]{paternain2015invariant}}\label{commutator}
The following commutator formulas hold:
\[
[X,\vnabla]=-\hnabla,\quad\quad [X,\hnabla]=R\vnabla,\quad\quad \hdiv\vnabla-\vdiv\hnabla=(d-1)X
\]
on $C^\infty(SM)$, and
\[
[X,\vdiv]=-\hdiv,\quad\quad [X,\hdiv]=-\vdiv R
\]
on $\mathcal{Z}$.
\end{lemma}

The following Pestov identity plays a crucial role in the proof of the main result.
\begin{proposition}{\cite[Proposition 2.2 and Remark 2.3]{paternain2015invariant}}
For any $u\in C^\infty(SM)$ with $u\vert_{\partial(SM)}=0$,
\begin{equation}\label{pestov}
\|\vnabla X u\|^2=\|X\vnabla u\|^2-(R\vnabla u,\vnabla u)+(d-1)\|Xu\|^2.
\end{equation}
\end{proposition}
For the proof of the above identity, as well as the integral identities below, one needs to repeatedly use integration by parts: for $u\in C^\infty(SM)$ and $Z\in\mathcal{Z}$ we have
\[
(\vnabla u,Z)=-(u,\vdiv Z);
\]
if $u,w\in C^\infty(SM)$ and $Z\in\mathcal{Z}$ with $u\vert_{\partial(SM)}=0$, we have
\[
(Xu,w)=-(u,Xw),\quad\quad
(\hnabla u,Z)=-(u,\hdiv Z);
\]
if $Z,W\in \mathcal{Z}$ with $Z\vert_{\partial (SM)}=0$, we have
\[
(XZ,W)=-(Z,XW).
\]
For more details, we refer to \cite{paternain2015invariant}.

Consider the vertical Laplacian
\[
\Delta:=-\vdiv\vnabla:\,C^\infty(SM)\rightarrow C^\infty(SM).
\]
Since $\Delta u(x,v)$ coincides with the Laplacian of the function $v\mapsto u(x,v)$ on the manifold $(S_xM,g_x)$, the Hilbert space $L^2(SM)$ admits an orthogonal decomposition
\[
L^2(SM)=\bigoplus_{k\geq 0}H_k(SM),
\]
where $\Delta u_k=k(k+d-2)u_k$ for any $u_k\in \Omega_k:=H_k(SM)\cap C^\infty(SM)$. For any $u\in L^2(SM)$, we write this decomposition as
\[
u=\sum_{k=0}^\infty u_k,\quad u_k\in H_k(SM).
\]

Note that $\Omega_k$ can be identified with the smooth \textit{trace-free} symmetric $k$-tensors. Therefore for $f$ associated with a smooth symmetric $m$-tensor, the tensor decomposition simplifies to
\[
f=\sum_{k=0}^{[m/2]}f_{m-2k}\in\bigoplus_{k=0}^{[m/2]}\Omega_{m-2k}.
\]
For more details see \cite{dairbekov2011conformal}.
We have
\begin{equation}\label{est_f}
\begin{split}
\|\vnabla f\|^2=(\vnabla f,\vnabla f)=(\Delta f,f)&=\sum_{k=0}^{[m/2]}(m-2k)(m-2k+d-2)\|f_{m-2k}\|^2\\
&\leq m(m+d-2)\sum_{k=0}^{[m/2]}\|f_{m-2k}\|^2\\
&= m(m+d-2)\|f\|^2.
\end{split}
\end{equation}
The above inequality resembles \cite[Lemma 5.2]{dairbekov2006integral} (cf. \cite[Lemma 4.5.3]{Shara} also).
Since $f$ is divergence-free, we have the following identity:
\begin{lemma}\label{divergencefree}
If $f$ is a smooth symmetric $m$-tensor satisfying $\delta f=0$, then
\begin{equation}\label{div_identity}
\hdiv\vnabla f+mXf=0.
\end{equation}
\end{lemma}
The proof of this lemma follows from tedious calculations, which will be given in the appendix.

\begin{proposition}
For $u$ satisfying $Xu=-f$ with $\delta f=0$ and $u\vert_{\partial(SM)}=0$, we have
\begin{equation}\label{eq2}
\|X\vnabla u\|^2-\|\vnabla Xu\|^2-\|\hnabla u\|^2-2m\|Xu\|^2=0.
\end{equation}
\end{proposition}
\begin{proof}
We start with the identity
\[
X\vnabla u=\vnabla Xu-\hnabla u,
\]
which follows Lemma \ref{commutator}.
Then, using \eqref{div_identity}, we have
\[
\begin{split}
\|X\vnabla u\|^2=&\|\vnabla Xu\|^2+\|\hnabla u\|^2-2(\vnabla Xu,\, \hnabla u)\\
=&\|\vnabla Xu\|^2+\|\hnabla u\|^2+2(\hdiv\vnabla Xu,u)\\
=&\|\vnabla Xu\|^2+\|\hnabla u\|^2-2m(XXu,u)\\
=&\|\vnabla Xu\|^2+\|\hnabla u\|^2+2m\|Xu\|^2.
\end{split}
\]
\end{proof}

We will also use
\begin{proposition}
For $u$ satisfying $Xu=-f$ with $\delta f=0$ and $u\vert_{\partial(SM)}=0$, we have
\begin{equation}\label{identity3}
\|\hnabla u\|^2+(d-1+2m)\|Xu\|^2-(R\vnabla u,\,\vnabla u)=0.
\end{equation}
\end{proposition}
\begin{proof}
The identity \eqref{identity3} follows directly from \eqref{pestov} and \eqref{eq2}. But we still give a ``from scratch" proof here. We start from \eqref{div_identity} and calculate using the commutator formulas in Lemma \ref{commutator}:
\begin{equation*}
\begin{split}
0=&(u,\, \hdiv\vnabla f+mXf)\\
=&-(u,\, \hdiv\vnabla Xu+mXXu)\\
=&(\hnabla u,\, \vnabla Xu)+m(Xu,\,Xu)\\
=&(\hnabla u, \, X\vnabla u +\hnabla u)+m(Xu,\,Xu)\\
=&\|\hnabla u\|^2+m\|Xu\|^2+(\hnabla u,\, X\vnabla u)\\
=&\|\hnabla u\|^2+m\|Xu\|^2-(X\hnabla u,\, \vnabla u)\\
=&\|\hnabla u\|^2+m\|Xu\|^2-(\hnabla Xu,\, \vnabla u)-(R\vnabla u,\,\vnabla u)\\
=&\|\hnabla u\|^2+m\|Xu\|^2+(\vdiv\hnabla Xu,\, u)-(R\vnabla u,\,\vnabla u)\\
=&\|\hnabla u\|^2+m\|Xu\|^2+(\hdiv\vnabla Xu,\, u)+(d-1)(Xu,\, Xu)-(R\vnabla u,\,\vnabla u)\\
=&\|\hnabla u\|^2+m\|Xu\|^2-(mX Xu,\, u)+(d-1)(Xu,\, Xu)-(R\vnabla u,\,\vnabla u)\\
=&\|\hnabla u\|^2+(d-1+2m)\|Xu\|^2-(R\vnabla u,\,\vnabla u).
\end{split}
\end{equation*}
\end{proof}
\begin{remark}
Notice that the formulas \eqref{est_f}, \eqref{div_identity}, \eqref{eq2}, \eqref{identity3} look different than their counterparts \textnormal{Lemma} $5.2$, \textnormal{Lemma} $5.1\,(2)$, $(5.14)$, $(5.6)$ in \cite{dairbekov2006integral}. This is because that we use the operators $\vnabla,\hnabla,\vdiv,\hdiv$ as defined in \cite{paternain2015invariant}, which are different than those used in \cite{dairbekov2006integral}.
\end{remark}

We also need to following lemma (cf. \cite[Remark 11.3]{paternain2015invariant} and \cite[Lemma 5.3]{dairbekov2006integral}).
\begin{lemma}\label{lemma1}
Assume $(M,g)$ is free of $\beta$-conjugate points, then
\[
\|XZ\|^2-\beta(RZ,Z)\geq 0
\]
for every $Z\in\mathcal{Z}$ with $Z\vert_{\partial(SM)}=0$. The equality holds only when $Z=0$.
\end{lemma}

\section{Proof of Theorem \ref{thm1} and Theorem \ref{thm2}}
Now we are ready to present the proof of the main theorems. We start with the transport equation \eqref{transport_eq} with $\delta f=0$, and show that $f=0$.
\begin{proof}[Proof of Theorem \ref{thm1}]
Take $\eqref{eq2}+\eqref{identity3}\times\beta^{(1)}_{d,m}$, we obtain
\[
\begin{split}
0=&\|X\vnabla u\|^2-\beta^{(1)}_{d,m} (R\vnabla u,\,\vnabla u)+(\beta^{(1)}_{d,m}-1)\|\hnabla u\|^2-\|\vnabla Xu\|^2-(2m-\beta^{(1)}_{d,m}(2m+d-1))\|Xu\|^2\\
=&\|X\vnabla u\|^2-\beta^{(1)}_{d,m}(R\vnabla u,\,\vnabla u)+(\beta^{(1)}_{d,m}-1)\|\hnabla u\|^2-\|\vnabla Xu\|^2+m(m+d-2)\|Xu\|^2.
\end{split}
\]
Notice that $\beta^{(1)}_{d,m}-1\geq 0$ and recall that (cf. \eqref{est_f})
\[
m(m+d-2)\|Xu\|^2-\|\vnabla Xu\|^2\geq 0.
\]
We have
\begin{equation}\label{ineqR}
\|X\vnabla u\|^2-\beta^{(1)}_{d,m}(R\vnabla u,\,\vnabla u)\leq 0,
\end{equation}
and therefore $\vnabla u=0$ by Lemma \ref{lemma1} if $(M,g)$ is free of $\beta^{(1)}_{d,m}$-conjugate points. Then using \eqref{identity3}, we get $f=Xu=0$.
\end{proof}

Next, we improve the result to Theorem \ref{thm2} using one more trick.

\begin{proof}[Proof of Theorem \ref{thm2}]
We will use the following inequality:
\begin{equation}\label{ineq3}
\begin{split}
0\leq \|\gamma\vnabla Xu+\hnabla u\|^2=&\gamma^2\|\vnabla Xu\|^2+\|\hnabla u\|^2+2\gamma( \vnabla Xu,\hnabla u)\\
=&\gamma^2\|\vnabla Xu\|^2+\|\hnabla u\|^2-2\gamma (\hdiv\vnabla Xu,u)\\
=&\gamma^2\|\vnabla Xu\|^2+\|\hnabla u\|^2-2\gamma m\|X u\|^2.
\end{split}
\end{equation}
for any real number $\gamma$.
Taking $\eqref{eq2}+\eqref{identity3}\times\beta-\eqref{ineq3}\times (\beta-1)$ with $\beta\geq 1$, we obtain
\[
0\geq\|X\vnabla u\|^2-\beta (R\vnabla u,\,\vnabla u)-((\beta-1)\gamma^2+1)\|\vnabla Xu\|^2+(-2m+2m\beta+d\beta-\beta+2\gamma (\beta-1)m)\|Xu\|^2
\]
for any real $\gamma$ and $\beta\geq 1$.
Notice that if
\[
m(m+d-2)((\beta-1)\gamma^2+1)\leq (-2m+2m\beta+d\beta-\beta+2\gamma (\beta-1)m),
\]
or equivalently,
\[
\beta\geq 1+\frac{m(m+d-2)-d+1}{2m\gamma-\gamma^2m(m+d-2)+2m+d-1}:=\beta_{d,m}(\gamma),
\]
we have
\[
-((\beta-1)\gamma^2+1)\|\vnabla Xu\|^2+(-2m+2m\beta+d\beta-\beta+2\gamma (\beta-1)m)\|Xu\|^2\geq 0.
\]
Here $2m\gamma-\gamma^2m(m+d-2)+2m+d-1$ (which is quadratic in $\gamma$) needs to be positive.
Notice that $\beta_{d,m}(\gamma)\in[1,\infty)$ attains the minimal value $\frac{m(m+d-1)}{2m+d-2}=\beta^{(2)}_{d,m}$ at $\gamma=\frac{1}{m+d-2}:=\gamma_{d,m}$. Thus
\[
\|X\vnabla u\|^2-\beta^{(2)}_{d,m}(R\vnabla u,\,\vnabla u)\leq 0.
\]
Notice that the above inequality improves \eqref{ineqR}. Arguing as in the proof of Theorem \ref{thm1}, one can finish the proof.
\end{proof}

\appendix

\section{Proof of Lemma \ref{divergencefree}}
Let $x=(x^1,x^2,\cdots, x^d)$ be local coordinates in $M$, and $(x,y)$ be the associated local coordinates on $TM$, where the tangent vector looks like $y=y^i\frac{\partial}{\partial x^i}$. In the local coordinates, the divergence operator $\delta$ is defined as
\[
(\delta f)_{i_1i_2\cdots i_{m-1}}:=f_{i_1i_2\cdots i_m;\,j}g^{i_mj}=(\partial_{x^j}f_{i_1i_2\cdots i_m}-\sum_{s=1}^m\Gamma^p_{ji_s}f_{i_1\cdots i_{s-1}pi_{s+1}\cdots i_m})g^{i_{m}j}.
\]
Thus $\delta f=0$ implies
\[
\begin{split}
0=\delta f(x,v)=(\delta f)_{i_1i_2\cdots v_{m-1}}v^{i_1}v^{i_2}\cdots v^{i_{m-1}}=&\partial_{x^j}f_{i_1i_2\cdots i_{m-1}i_m}g^{i_mj}v^{i_1}v^{i_2}\cdots v^{i_{m-1}}\\
&-\Gamma_{kj}^pg^{kj}f_{i_1i_2\cdots i_{m-1}p}v^{i_1}v^{i_2}\cdots v^{i_{m-1}}\\
&-(m-1)\Gamma^p_{ji_{m-1}}f_{i_1i_2\cdots i_{m-2}pi_{m}}g^{i_mj}v^{i_1}v^{i_2}\cdots v^{i_{m-1}}.
\end{split}
\]
For $u\in C^\infty(SM)$, we have the following form of the vector fields $X,\vnabla,\hnabla$ as in \cite[Appendix A]{paternain2015invariant}:
\[
\begin{split}
&Xu=v^j\delta_ju,\\
&\vnabla u=(\partial^ku)\partial_{x^k},\\
&\hnabla u=(\delta^ju-(v^k\delta_ku)v^j)\partial_{x^j},
\end{split}
\]
where
\[
\begin{split}
&\delta_ju=(\partial_{x^j}-\Gamma^\ell_{jk}y^k\partial_{y^\ell})u(x,\frac{y}{|y|})\vert_{SM}=:\delta_{x^j}u(x,\frac{y}{|y|})\vert_{SM},\\
&\partial_ku=\partial_{y^k}u(x,\frac{y}{|y|})\vert_{SM}.
\end{split}
\]
Here $|y|=|y|_g=\sqrt{g_{\alpha\beta}y^\alpha y^\beta}$.
On $Z=Z^j\partial_{x^j}\in\mathcal{Z}$, the operators $X,\vdiv,\hdiv$ are given as
\[
\begin{split}
&XZ(x,v)=(XZ^j)\partial_{x^j}+\Gamma_{jk}^\ell v^jZ^k\partial_{x^\ell},\\
&\vdiv Z=\partial_jZ^j,\\
&\hdiv Z=(\delta_j+\Gamma_j)Z^j.\\
\end{split}
\]
Here $\Gamma_j=\Gamma_{jk}^k$.

Using the above expressions in local coordinates, we calculate
\[
\begin{split}
\partial_{y^\ell}(f_{i_1i_2\cdots i_m}|y|^{-m}y^{i_1}y^{i_2}\cdots y^{i_m})=&mf_{i_1i_2\cdots i_{m-1}\ell}|y|^{-m}y^{i_1}y^{i_2}\cdots y^{i_{m-1}}\\
&-mf_{i_1i_2\cdots i_m}|y|^{-m-2}y^{i_1}y^{i_2}\cdots y^{i_m}y_\ell,
\end{split}
\]
and
\[
\begin{split}
\delta_{x^j}(f_{i_1i_2\cdots i_m}|y|^{-m}y^{i_1}y^{i_2}\cdots y^{i_m})=&\partial_{x^j}f_{i_1i_2\cdots i_m}|y|^{-m}y^{i_1}y^{i_2}\cdots y^{i_m}\\
&-m\Gamma_{jk}^\ell f_{i_1i_2\cdots i_{m-1}\ell}|y|^{-m}y^{i_1}y^{i_2}\cdots y^{i_{m-1}}y^k\\
&+m\Gamma_{jk}^\ell f_{i_1i_2\cdots i_m}|y|^{-m-2}y^{i_1}y^{i_2}\cdots y^{i_m}y_\ell y^k\\
&-\frac{m}{2}|y|^{-m-2}\partial_{x^k}g_{\alpha\beta}y^{\alpha}y^\beta(f_{i_1i_2\cdots i_m}y^{i_1}y^{i_2}\cdots y^{i_m}).
\end{split}
\]
Then we have
\[
\begin{split}
Xf=v^j\delta_jf=&\partial_{x^j}f_{i_1i_2\cdots i_m}v^{i_1}v^{i_2}\cdots v^{i_m}v^j-m\Gamma_{jk}^\ell f_{i_1i_2\cdots i_{m-1}\ell}v^{i_1}v^{i_2}\cdots v^{i_{m-1}}v^kv^j\\
&+m\Gamma_{jk}^\ell f_{i_1i_2\cdots i_m}v^{i_1}v^{i_2}\cdots v^{i_m}v_\ell v^kv^j-\frac{m}{2}f_{i_1i_2\cdots i_m}\partial_{x_k}g_{\alpha\beta}v^{\alpha}v^\beta v^{i_1}v^{i_2}\cdots v^{i_m}v^k,
\end{split}
\]
and
\[
\vnabla f=(mf_{i_1i_2\cdots i_{m-1}\ell}v^{i_1}v^{i_2}\cdots v^{i_{m-1}}g^{\ell k}-mf_{i_1i_2\cdots i_m}v^{i_1}v^{i_2}\cdots v^{i_m}v^k)\partial_{x^k}.
\]
Denote $\vnabla f=Z^k\partial_{x^k}$, where $Z^k=mf_{i_1i_2\cdots i_{m-1}\ell}v^{i_1}v^{i_2}\cdots v^{i_{m-1}}g^{\ell k}-mf_{i_1i_2\cdots i_m}v^{i_1}v^{i_2}\cdots v^{i_m}v^k$. Then we calculate
\[
\begin{split}
\delta_jZ^j=&m\partial_{x^j}f_{i_1i_2\cdots i_{m-1}\ell}v^{i_1}v^{i_2}\cdots v^{i_{m-1}}g^{\ell j}+mf_{i_1i_2\cdots i_{m-1}\ell}v^{i_1}v^{i_2}\cdots v^{i_{m-1}}\partial_{x_j}g^{\ell j}\\
&-m\partial_{x^j}f_{i_1i_2\cdots i_m}v^{i_1}v^{i_2}\cdots v^{i_m}v^j-m(m-1)\Gamma_{jq}^pv^qf_{i_1i_2\cdots i_{m-2}p\ell}v^{i_1}v^{i_2}\cdots v^{i_{m-2}}g^{\ell j}\\
&+m(m-1)\Gamma_{jq}^pv^qf_{i_1i_2\cdots i_{m-1}\ell}v^{i_1}v^{i_2}\cdots v^{i_{m-1}}g^{\ell j}v_p+m^2\Gamma_{jq}^pv^qf_{i_1i_2\cdots i_{m-1}p}v^{i_1}v^{i_2}\cdots v^{i_{m-1}}v^j\\
&+m\Gamma_{jq}^jv^qf_{i_1i_2\cdots v_m}v^{i_1}v^{i_2}\cdots v^{i_m}-m(m+1)\Gamma_{jq}^pv^qf_{i_1i_2\cdots i_m}v^{i_1}v^{i_2}\cdots v^{i_m}v^jv_p\\
&-\frac{m(m-1)}{2}f_{i_1i_2\cdots i_{m-1}\ell}v^{i_1}v^{i_2}\cdots v^{i_{m-1}}g^{\ell j}(\partial_{x^j}g_{\alpha\beta})v^\alpha v^\beta\\
&+\frac{m(m+1)}{2}f_{i_1i_2\cdots i_m}v^{i_1}v^{i_2}\cdots v^{i_m}v^j(\partial_{x^j}g_{\alpha\beta})v^\alpha v^\beta,
\end{split}
\]
and
\[
\Gamma_jZ^j=m\Gamma_jf_{i_1i_2\cdots i_{m-1}\ell}v^{i_1}v^{i_2}\cdots v^{i_{m-1}}g^{\ell j}-m\Gamma_jf_{i_1i_2\cdots i_m}v^{i_1}v^{i_2}\cdots v^{i_m}v^j.
\]
Using the facts $\partial_{x^j}g_{\alpha\beta}=\Gamma_{\alpha j}^ig_{i\beta}+\Gamma_{\beta j}^ig_{i\alpha}$ and $\partial_{x^j}g^{j\ell}=-\Gamma_pg^{p\ell}-\Gamma_{kp}^\ell g^{kp}$, we obtain
\[
\begin{split}
&\hdiv\vnabla f+mXf\\
=&-m\Gamma_{jq}^pv^qf_{i_1i_2\cdots i_m}v^{i_1}v^{i_2}\cdots v^{i_m}v^jv_p+mf_{i_1i_2\cdots i_m}v^{i_1}v^{i_2}\cdots v^{i_m}v^j\Gamma_{\alpha j}^ig_{i\beta}v^\alpha v^\beta\\
&+m\Gamma_jf_{i_1i_2\cdots i_{m-1}\ell}v^{i_1}v^{i_2}\cdots v^{i_{m-1}}g^{\ell j}-mf_{i_1i_2\cdots i_{m-1}\ell}v^{i_1}v^{i_2}\cdots v^{i_{m-1}}(\Gamma_pg^{\ell p}+\Gamma_{pk}^\ell g^{kp})\\
&+m\partial_{x^j}f_{i_1i_2\cdots i_{m-1}\ell}v^{i_1}v^{i_2}\cdots v^{i_{m-1}}g^{\ell j}-m(m-1)\Gamma_{jq}^pv^qf_{i_1i_2\cdots i_{m-2}p\ell}v^{i_1}v^{i_2}\cdots v^{i_{m-2}}g^{\ell j}\\
=&m(\delta f)_{i_1i_2\cdots i_{m-1}}v^{i_1}v^{i_2}\cdots v^{i_{m-1}}\\
=&0.
\end{split}
\]

\bibliographystyle{abbrv}
\bibliography{biblio}
\end{document}